\documentclass[a4paper, 11pt]{amsart}
\usepackage[utf8]{inputenc}
\usepackage{amssymb, mathrsfs, amsfonts, amsmath}
\usepackage{mathtools} 
\usepackage{setspace}
\usepackage[left=2.7cm, right=2.7cm, top=3cm, bottom=3cm]{geometry} 
\usepackage{hyperref, latexsym}
\title[]{A note on discrete Heisenberg uniqueness pairs for the parabola} 
\author{Felipe Gon\c calves and Jo\~ao P. G. Ramos}
\usepackage{xcolor}

\newtheorem*{defn}{Definition}
\newtheorem{theorem}{Theorem}
\newtheorem{lemma}{Lemma}

\newtheorem{rem}{Remark}

\newcommand{\R}{\mathbb{R}}
\newcommand{\N}{\mathbb{N}}
\newcommand{\Z}{\mathbb{Z}}
\newcommand{\mmd}{\mathrm{d}}

\usepackage{mathtools}
\mathtoolsset{showonlyrefs}
\newcommand{\wh}{\widehat}
\renewcommand{\P}{\mathcal{P}}
\newcommand{\F}{\mathcal{F}}
\newcommand{\ep}{\varepsilon}

\begin{document}

\begin{abstract}
We discuss on Heisenberg uniqueness pairs for the parabola given by discrete sequences along straight lines. Our method consists in linking the problem at hand with recent uniqueness results for the Fourier transform.
\end{abstract}

\maketitle

\section{Introduction} For an algebraic curve $\Gamma \subset \R^2$ and a set $\Lambda \subset \R^2$, we say that the pair $(\Gamma,\Lambda)$ is a \emph{uniqueness pair} (or Heisenberg uniqueness pair) if whenever 
$\mu$ is a borel complex measure on $\Gamma$, absolutely continuous with respect to the arc-length measure on $\Gamma$, $\mu\in L^1(\Gamma)$ and $\widehat{\mu}|_{\Lambda} = 0,$ then $\mu \equiv 0$. Here we define the
Fourier transform \footnote{Also called the Fourier extension of $\mu$ in Fourier restriction theory, or the characteristic function of the measure $\mu$ in probability theory.} of a measure $\mu$ on $\R^2$ by
\begin{align}\label{def:fourierextensionmu}
\wh \mu(\xi,\eta) = \int_{\R^2} e^{-2\pi i (\xi x + \eta y)} \, \mmd \mu(x,y).
\end{align}

This concept is inspired in the classical Heisenberg uncertainty principle for the Fourier transform, but it has been explored in more depth since the pioneering work of Hedenmalm and Montes-Rodriguez \cite{HMR}. 
There, the authors analyze some classical cases of Heisenberg uniqueness pairs with $\Gamma$ being a conic section. In particular, their main results concern the hyperbola 
$\Gamma = \{(x,y) \in \R^2 \colon xy = 1\}$ and lattice-crosses of the form $\Lambda_{\alpha,\beta} = ( \alpha \Z \times \{0\} ) \cup (\{0\} \times \beta \Z)$, where they prove
$(\Gamma, \Lambda_{\alpha,\beta})$ is a uniqueness pair if and only if $\alpha \beta \le 1$. 

Following their work, subsequent progress has been made on several of the conic section cases and other related questions \cite{GS, JK, HMR2, CMHMR}. We highlight the \emph{circle} case, in which 
the works of Sj\"olin \cite{Sjoe1} and Lev \cite{Lev} prove that if  $\Lambda=L_1 \cup ... \cup L_N$, where $L_1,...,L_N$ are straight lines with a common intersection point, then $(\mathbb{S}^1,\Lambda)$ is not a uniqueness pair exactly when all intersection angles of the lines $L_1,...,L_N$ are rational multiples of $\pi$ (see also \cite{FBGJ} for a different approach).

In this note, however, we shall be concerned with the specific case of  the parabola
$$
\P = \{(x,y) \in \R^2 \colon x^2 = y\}.
$$
A simple computation shows that any finite measure $\mu$ supported on $\P$  has the property that $\widehat{\mu}$ solves the \emph{Schr\"odinger equation}
\[
i \partial_y \widehat{\mu} + \partial_x^2 \widehat{\mu} = 0.
\]
In that regard, we mention the work of Sj\"olin \cite{Sjoe2}, where the author analyzes when sets of the form $\Lambda = E_1 \cup E_2$ are so that $(\P,\Lambda)$ is a uniqueness pair and 
$E_1,E_2$ are two sets of positive length contained in distinct straight lines. See also \cite{GJ} and \cite{Vieli} for generalizations to higher dimensions. 

\begin{defn}[Weak uniqueness pair] We say that a pair $(\Gamma,\Lambda) \subset \R^2\times \R^2$ is a \emph{weak} uniqueness pair if there is $C=C(\Gamma,\Lambda)>0$ such that whenever $\mu$ is a complex Borel measure with $\text{supp}(\mu) \subset \Gamma$ and absolutely continuous with respect to the arc-length measure of $\Gamma$, with $\mu\in L^1(\Gamma)$, and satisfying
$$
\wh{\mu}|_\Lambda=0 \ \ \text{and} \ \ (1+|x|^C)|\wh{ \mu}(x,0)| \in L^1(\R),
$$
then $\mu\equiv 0$. 
\end{defn}


As highlighted before, we aim to look at this concept for the parabola case $\Gamma=\P$ and a set $\Lambda$ consisting of discrete points along two or three lines. This can be seen, for instance, as a natural complement to the results of Hedenmalm and Montes-Rodriguez \cite{HMR}, replacing the hyperbola with the parabola. 

Our main result is the following.

\begin{theorem}\label{thm:teo-1} Let $\P = \{(t,t^2): \, t \in \R\}$ be the parabola. 

\noindent $\bullet \ $ $($Two parallel lines$)$ Let
 $$
 \Lambda = \{(\pm c_1 \, n^{\alpha},0):\, n \in \N\} \cup \{(\pm c_2\, n^{\beta},1):\, n \in \N\}
 $$
 for some $c_1,c_2>0$, where $\N=\{0,1,2,3,...\}$. Then $(\P,\Lambda)$ is a weak uniqueness pair if $(\alpha,\beta) \in A,$ where $A$ is defined as 
 
     \begin{align*} A  = & \left\{(\alpha,\beta) \in (0,1)^2
 	\text{ and either }\alpha < \frac{(1-\beta)^2}{2-\beta}
 	\text{ or } \beta < \frac{(1-\alpha)^2}{2-\alpha}\right\}.
 \end{align*}

\noindent $\bullet \ $ $($Three lines intersecting at a point$)$ Let
 $$
 \Lambda = \{\pm (c_1 \, n^{-\alpha}, a \,c_1 \,n^{-\alpha}): n\in \N^*\} \cup \{ \pm (c_2\, n^{\beta/2},0): n \in \N\} \cup  \{\pm (c_3\,n^{-\gamma}, d \,c_3 \,n^{-\gamma}): n\in \N^*\}
 $$
for some $c_1,c_2,c_3>0$, with $a\neq 0$, $d\neq 0$ and $a\neq d$. Then $(\P,\Lambda)$ is a weak uniqueness pair if $(\alpha,\beta),(\gamma,\beta) \in A$.
\end{theorem}

\begin{figure}
	\centering
	\includegraphics[scale=0.13]{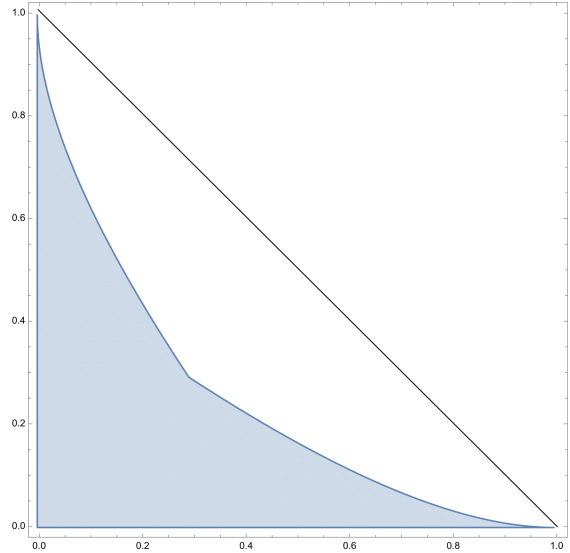}
	\caption{In {\color{blue} blue}, the region $A$ described in Theorem \ref{thm:teo-1}.}
\end{figure}

We prove that $(\P,\Lambda)$ is a weak uniqueness pair, with any $C>2$ in the definition of weak uniqueness pairs, for the two lines case, and with any
$$
C>2\max\bigl(2,1+3\lceil \gamma/(1-\gamma) \rceil,\, 1+3\lceil \alpha/(1-\alpha) \rceil\bigr)
$$
for three lines. Using standard Galilean symmetries of Schrodinger's equation is easy to see that an analogue of our result still follows for \emph{any} two horizontal lines and \emph{any} three lines with a single common intersection and one of them horizontal. Using pseudo-conformal inversion
$$
\wh \mu (x,y) \mapsto (a+by)^{-1/2} e^{ibx^2/(4(a+by))}\wh \mu\Bigl(\frac{x}{a+by},\frac{c+dy}{a+by}\Bigr), \quad ad-bc=1  
$$
the intersection can actually be at infinity, that is, the three lines case still is true if
$$
\Lambda=\{\pm (a, c_1 \,n^{\alpha}): n\in \N^*\} \cup \{ \pm (c_2\, n^{\beta/2},0): n \in \N\} \cup  \{\pm (d, c_3 \,n^{\gamma}): n\in \N^*\}
$$
for $a\neq d$. 

Notice that our theorem extends some of the previous results by Sj\"olin \cite{Sjoe1} for the parabola to the case of discrete sequences, rather than large measure sets. Theorem \ref{thm:teo-1} has an elegant short proof, connecting Heisenberg uncertainty principles to Fourier uniqueness pairs as in \cite{RamosSousa}.

The rest of this manuscript is devoted to the proof of Theorem \ref{thm:teo-1}, together with a few comments on its proof and possible generalizations. 

\section{Proof of Theorem \ref{thm:teo-1}} 

To prove Theorem \ref{thm:teo-1} we will invoke the main result of \cite{RamosSousa}, which we state for convenience. In what follows $\mathcal{S}(\R)$ stands for the Schwartz class.

\begin{theorem}[Main theorem in \cite{RamosSousa}]\label{thm:fourier-uniqueness} 
Let $f\in \mathcal{S}(\R)$ and suppose that $(\alpha,\beta) \in A$, where $A$ is as in Theorem \ref{thm:teo-1}. 
If for some $c_1,c_2>0$ we have $f(\pm c_1 n^{\alpha})=\widehat{f}(\pm c_2 n^{\beta})=0$ for every $n\in \N$ sufficiently large, then $f\equiv 0$.
\end{theorem}

Theorem \ref{thm:fourier-uniqueness} is directly related to the Fourier interpolation formula by Radchenko and Viazovska \cite{RadchenkoViazovska}. Indeed, they proved in \cite{RadchenkoViazovska} that any function $f \in \mathcal{S}_{even}(\R)$ may be determined -- and in fact \emph{interpolated} -- from the set $\{f(\sqrt{n})\}_{n \ge 0} \cup \{\widehat{f}(\sqrt{n})\}_{n\ge0},$ and Theorem \ref{thm:fourier-uniqueness} deals with the same uniqueness question for more general powers of integers. 

 To use Theorem \ref{thm:fourier-uniqueness} we need to pass from low to infinite regularity, and so we define
$$\F^C=\bigl\{f \in L^1(\R): (1+|x|^C)\bigl(|f(x)| + |\widehat{f}(x)|\bigr) \in L^1(\R)\bigr\},$$ where $C\geq 0$. Note $\wh f$ and $f$ are continuous if $f\in \F^C$. We have the following lemma. 

\begin{lemma}\label{thm:reduction}  
Let $\alpha,\beta>0$ with $\alpha+\beta<1$ and define 
$$
C(\alpha,\beta):=1+\Bigl\lfloor 2\frac{\min(\alpha,\beta)}{1-\alpha-\beta}\Bigr\rfloor.
$$ 
Suppose $f \in \F^C(\R)$, for some $C> C(\alpha,\beta)$, and
\begin{equation}\label{eq:fourier-uniqueness} 
f(\pm c_1 n^{\alpha}) = \widehat{f}(\pm c_2 n^{\beta}) = 0, \, \forall \, n \in \N,
\end{equation}
for some $c_1,c_2>0$. Then $f \in \mathcal{S}(\R)$. 
\end{lemma}

\begin{proof} For the proof of this lemma, we need to recall some facts from \cite{RamosSousa}. Firstly, we notice that, by the proof of \cite[Lemma~5 and ~6]{RamosSousa}, the following property holds: If $f \in \F^0(\R)$ satisfies \eqref{eq:fourier-uniqueness} for $\alpha, \beta \geq 0$ and $\alpha + \beta < 1,$ then the estimates\footnote{Estimates \eqref{eq:fk} and \eqref{eq:fhatj} also hold if the vanishing condition \eqref{eq:fourier-uniqueness} is true only for $n\geq n_0$.}
\begin{align}\label{eq:fk}
        |f(x)|  &\leq B_{\alpha,k} I_k(\widehat{f})|x|^{\left(\frac{\alpha -1}{\alpha}\right)k} \quad (\text{for } \ |x|\geq X_{\alpha,k})  \\
  \label{eq:fhatj}          |\widehat{f}(\xi)|  &\leq B_{\beta,j} I_{j}(f)|\xi|^{\left(\frac{\beta -1}{\beta}\right)j} \quad \; (\text{for } \ |x|\geq X_{\beta,j})
\end{align}
hold whenever $I_k(\wh f), I_j(f)<\infty$, where we have used the shorthand 
$$I_k(h) = \int_{\R} |x|^k |h(x)| \, \mmd x$$
and $X_{\alpha,k},B_{\alpha,k}$ are some explicit positive constants. In particular, if $f\in \F^C(\R)$ then \eqref{eq:fk} and \eqref{eq:fhatj} hold for $k,j\leq C$. What is important to note is that these estimates are self-improving. For instance, given $\ep>0,$ a simple computation shows that if  \eqref{eq:fhatj} holds for some $j\geq 1,$ then $I_k(\widehat{f}) < +\infty$ for $1 \leq k\leq j\frac{1-\beta}{\beta} - 1-\ep$, so that \eqref{eq:fk} holds for those $k'$s. Similarly, if \eqref{eq:fk} holds with some $k\geq 1$ then \eqref{eq:fhatj} holds for $1 \leq j\leq k\frac{1-\alpha}{\alpha} - 1-\ep$. Thus, we want to make sure that 
$$
j+1 \leq \Bigr\lfloor j\frac{1-\beta}{\beta} - 1-\ep\Bigr\rfloor\frac{1-\alpha}{\alpha}-1-\ep.
$$
Lower bounding $\lfloor x\rfloor> x-1$, the above inequality is satisfied if
$$
j\geq  (2+\ep)\beta/(1-\alpha-\beta)
$$
Noting that if \eqref{eq:fhatj} holds for all $j\geq 1$ then \eqref{eq:fk} holds for all $k\geq 1$ as well (and vice-versa), we deduce that we must have
$$
C \geq  \Bigl\lceil (2+\ep)\frac{\min(\alpha,\beta)}{1-\alpha-\beta}\Bigr\rceil \ \ \substack{{\longrightarrow} \\ {\small{\ep\to0}}} \ \  C(\alpha,\beta)
$$
Taking $\ep\to 0$ we conclude that if $f\in \F^C(\R)$ for some $C> C(\alpha,\beta)$ and $f$ satisfies \eqref{eq:fourier-uniqueness}, then $f\in \F^C(\R)$ for any $C\geq 0$. It is then not hard to show this implies $f\in \mathcal{S}(\R)$.
\end{proof}

\begin{rem}\label{remafterlemma}
We notice that estimates \eqref{eq:fk} and \eqref{eq:fhatj} still hold true if the vanishing condition \eqref{eq:fourier-uniqueness} holds only for $n\in \N\setminus R$, where $R$ is a finite set. Also, a simple computation shows
$$
\max_{(\alpha,\beta)\in A} \frac{\min(\alpha,\beta)}{1-\alpha-\beta} =\frac1{\sqrt{2}} = 0.7071...,
$$
and thus $2=\max_{(\alpha,\beta)\in A} C(\alpha,\beta)$. In particular, we conclude Theorem \ref{thm:fourier-uniqueness} is true if $f\in \F^{2+\delta}(\R),$ with $\delta > 0$ arbitrary, and the vanishing condition holds for all but finitely many $n\in \N$.
\end{rem}
Now we are ready to prove our main result.

\begin{proof}[\bf Proof of Theorem \ref{thm:teo-1}]  Throughout the proof we shall write 
	
	$$\int\varphi(x,y) \, \mmd\mu(x,y) = \int_{\R} \varphi(t,t^2) g(t) \mmd t,$$ 
	with $g\in L^1(\R).$
\noindent \underline{\emph{Two parallel lines}}.  We have
\begin{align}\label{id:muhatid}
\wh \mu(x,y) = \int_\R g(t)e^{-2\pi i (tx+t^2y)}\mmd t & = (2iy)^{-1/2} \int_\R \wh g(t+x) e^{-\pi t^2/(2iy)} \mmd t   \\
& = (2iy)^{-1/2} e^{i\pi x^2/(2y)}\int_\R \left(\wh g(t)e^{\pi i t^2/(2y)}\right) e^{\pi i t x/y} \mmd t,
\end{align}
where in the second identity we have used that the Fourier transform of $(2iy)^{-1/2}e^{-\pi t^2/(2iy)}$ is $e^{-2\pi i t^2y}$ (as a distribution). Note that $\wh \mu(x,0)=\wh g(x)$, hence the assumption $\wh \mu(x,0) \in L^1(\R)$ makes the calculation above correct. Now, if we define $f_y(t)=\wh g(t)e^{\pi i t^2/(2y)}$ then we see that 
\begin{align*}
\wh \mu(x,y)=(2iy)^{-1/2} e^{i\pi x^2/(2y)} \wh{f_y}(-x/(2y)), \quad y\neq 0.
\end{align*}
We obtain that $|\wh \mu(x,0)|=|f_1(x)|$ and $|\wh \mu(x,1)|=2^{-1/2} |\wh f_1(-x/2)|$. In particular, the vanishing condition $\wh \mu|_\Lambda=0$ with
$$
\Lambda = \{(\pm c_1 \, n^{\alpha},0):\, n \in \N\} \cup \{(\pm c_2\, n^{\beta},1):\, n \in \N\},
$$
which by symmetry we can assume $\alpha\geq \beta$, implies that $f_1(\pm c_1 n^\alpha)=\wh{f_1}(\pm {c_2}n^\beta/2)=0$ for all $n\in\N$. Note that if $\alpha \geq \beta$ and $(\alpha,\beta)\in A$ then $\beta \leq 1-1/\sqrt{2}=0.2928...$. Thus, if we assume that $(1+|x|^2)\wh \mu(x,0) \in L^1(\R)$ then $(1+|x|^2)f_1(x)\in L^1(\R)$ and we can apply estimate \eqref{eq:fhatj} to obtain
$$
|\wh f_1(\xi)| \leq B_{\beta,j} I_j(f_1) |x|^{j\frac{\beta-1}{\beta}} \quad (\text{for } |x|\geq X_{\beta,j})
$$
for $j\leq 2$. Since $2(\beta-1)/\beta+2<-1$ and $C\geq 2$ we conclude that $f_1\in \F^2(\R)$ and we can apply Lemma \ref{thm:reduction} and Remark \ref{remafterlemma} to conclude that $f_1 \in \mathcal{S}(\R)$. We can then use Theorem \ref{thm:fourier-uniqueness} to conclude that $f_1\equiv 0$. By uniqueness of solutions of the Schrödinger equation we must have $\wh \mu\equiv 0$, which implies $\mu \equiv 0$.

\noindent \underline{\emph{Three lines intersecting at a point}}.
We first use only the vanishing condition on
\[
\Lambda_{1} := \{\pm (c_1 n^{-\alpha}, a c_1 n^{-\alpha}): n \in \N^*\} \cup \{ \pm (c_2 n^{\beta/2},0): n \in \N\}.
\]
Setting $y=ax$ in identity \eqref{id:muhatid} we obtain
\begin{align*}
\wh \mu(x,ax) & = (2iax)^{-1/2} e^{i\pi x/(2a)} \int_\R \left(\wh g(t) e^{\pi i t/a}\right)e^{\pi i t^2/(2ax)} \mmd t \\ 
& = (2iax)^{-1/2} e^{i\pi x/(2a)}  \int_\R H(s) e^{\pi i s/(2ax)} \mmd s \\
& = (2iax)^{-1/2} e^{i\pi x/(2a)} \wh H(-1/(4ax)),
\end{align*}
where 
\[
H(s) = \begin{cases}
        \Bigl({\wh g(\sqrt{s}) e^{\pi i \sqrt{s}/a}+ \wh g(-\sqrt{s}) e^{-\pi i \sqrt{s}/a}}\Bigr)/({2\sqrt{s}}) & \, s> 0,\cr
        \quad \quad \quad \quad \quad \quad \quad \quad 0 & \, s \leq 0. 
       \end{cases}
\]
Since $\wh \mu(x,0)=\wh g(x) \in L^1(\R)$ the above computations are justified and $H\in L^1(\R)$. The vanishing condition of $\wh \mu$ on $\Lambda_1$ implies that $H(\pm {c_2}^2 \, n^{\beta})=\wh H(\pm\, n^{\alpha}/(4ac_1))=0$ for $n\in \N^*$. We assume now that $(1+|x|^C)\wh \mu(x,0) \in L^1(\R)$ with 
$
C=2\max\bigl(2,1+3\lceil \alpha/(1-\alpha)\rceil\bigr).
$
We can now apply estimate \eqref{eq:fhatj} to obtain
$$
|\wh H(\xi)| \leq B_{\alpha,j} I_j(H) |x|^{j\frac{\alpha-1}{\alpha}} \quad (|x|\geq X_{\alpha,j})
$$
for $j\leq C/2$. Since $C$ was chosen so that $\tfrac{C}2(\alpha-1)/\alpha+2 < -1$ and $C/2\geq 2$ we conclude that $H \in \F^2(\R)$. We can now use Lemma \ref{thm:reduction}, Remark \ref{remafterlemma} and Theorem \ref{thm:fourier-uniqueness} to conclude that $H\equiv 0$, that is, $\psi(x) := \wh g(x) e^{\pi i x/a}$ is odd. We can now repeat the argument for
\[
\Lambda_2 := \{\pm (c_3 n^{-\gamma}, d c_3 n^{-\gamma}): n \in \N^*\} \cup \{ \pm (c_2 n^{\beta/2},0): n \in \N\}
\]
(with a corresponding choice of $C$) and conclude that $\wh g(x) e^{\pi i x/d} = e^{\pi i x (1/d - 1/a)} \psi(x)$ is also odd. Setting $\alpha = 1/d - 1/a,$ we have 
\[
e^{-\pi i \alpha x} \psi(-x) = e^{\pi i \alpha x} \psi(x) = e^{\pi i \alpha x} \psi(-x)
\]
almost everywhere. Thus, for almost every $x$ in the support of $\psi,$ we must have $e^{-\pi i \alpha x} = e^{\pi i \alpha x}.$ But this happens only on a set of measure zero, which implies $\psi(x) = 0$ almost everywhere, and hence $\wh g \equiv 0$. We conclude as before that $\mu\equiv 0$.
\end{proof}

\section{Final Remarks}
Some comments are in order about Theorem \ref{thm:teo-1} and its proof. First of all, notice that the fact $\Lambda$ in the second part is constituted of discrete sets along three different lines is essential for our proof. Indeed, otherwise one can only conclude that
$\widehat{g}(\xi)e^{2\pi i r \xi}$ is odd for some $r \in \R$ and $\wh g(\pm c n^\beta)=0$ for all $n\in \N^*$.
It is not hard to notice that any odd smooth function $\wh g$ of sufficiently small support yields examples of such functions. 

Notice as well that we have \emph{always} considered at least one horizontal line. There are several possible reasons for such a choice. First, our very definition of weak uniqueness pair encompasses a condition inherent to the horizontal line case, and thus this choice becomes natural. On the other hand, generally speaking, uniqueness results for the 
Schr\"odinger equation usually assume some condition on the initial data, which, in this case, is translated into a horizontal line. Nevertheless, one might still inquire the case when the discrete set considered is in no such horizontal line. In this case, the basic feature of our proofs of relating the restriction of $\wh \mu$ on a line $L_1$ to its restriction on another line 
$L_2$ as Fourier transforms of one-dimensional functions is no longer available. In particular, relating $e^{it\partial_x^2}g(0)$ and $e^{it\partial_x^2}g(1)$ cannot be done through a simple process as in our proof, and we defer such an analysis for future work.

Finally, notice that relying only on the results in \cite{RamosSousa} gives us the range of exponentes allowed in Theorem \ref{thm:teo-1}, but one may wonder whether this is optimal. 
Indeed, in our proof itself we have used Lemma \ref{thm:reduction}, which is inspired by the recent work in progress by F. Nazarov and M. Sodin \cite{NazarovSodin}. In fact, in their work, the authors are able to extend the results in \cite{RamosSousa} to the maximal conjectured range. A version of their main results is as follows.

\begin{theorem}[Main theorem in \cite{NazarovSodin}]\label{thm:nazarov-sodin} Let $f \in \F^1(\R)$ be such that $f|_U = \widehat{f}|_V = 0,$ where the sets $U = \{u_j\}_{j \in \Z}, \, V = \{v_j\}_{j \in \Z}$ satisfy 
\begin{align*}
 \limsup_{|j| \to \infty} |u_j|^{p-1} (u_{j+1} - u_j) & < \frac{1}{2}, \cr 
 \limsup_{|j| \to \infty} |v_j|^{q-1}(v_{j+1} - v_j) & < \frac{1}{2},\cr 
\end{align*}
for some $p,q \in (1,+\infty), \frac{1}{p} + \frac{1}{q} = 1.$ Then $f \equiv 0.$ 
\end{theorem}

With this result in hands, it is not hard to see that we may substitute the condition in Theorem \ref{thm:teo-1} to simply $\alpha + \beta < 1$, in the two lines case, and  $\max\{\alpha,\gamma\} + \beta < 1$ in the three lines case. We believe that other techniques stemming from \cite{NazarovSodin} will be helpful in possibly removing the weak condition of Theorem \ref{thm:teo-1}.

\section*{Acknowledgements} 

We thank the anonymous referee for helpful comments. J.P.G.R. acknowledges financial support through ERC grant agreement No. 721675 ``Regularity and Stability in Partial Differential Equations (RSPDE)''.

\end{document}